\documentclass[12pt,letterpaper,reqno]{amsart}

\usepackage{amsthm,amssymb,mathrsfs}
\usepackage{amsmath}
\usepackage{tikz-cd}
\usepackage{tikz}
\usepackage{graphicx} 
\usepackage{xcolor}
\usepackage{pinlabel}
\usepackage{fourier}
\usepackage[breaklinks,colorlinks,citecolor=blue,linkcolor=blue,
 urlcolor=teal]{hyperref} 

\usepackage[margin=3cm]{geometry}

\usepackage{tikz}
\usepackage{circuitikz}
\usetikzlibrary{ decorations.markings}
\usetikzlibrary{arrows,arrows.meta,calc,decorations.markings,decorations.pathmorphing,decorations.pathreplacing,backgrounds,positioning,fit}
\definecolor{mygreen}{HTML}{278444}

\theoremstyle{plain}
\newtheorem{theorem}[equation]{Theorem} 
\newtheorem{proposition}[equation]{Proposition}
\newtheorem{lemma}[equation]{Lemma}
\newtheorem{corollary}[equation]{Corollary}

\theoremstyle{definition}
\newtheorem{definition}[equation]{Definition}

\newtheorem{remark}[equation]{Remark}

\newtheorem*{ack}{Acknowledgments}

\newcommand{\co}{\colon \thinspace}
\newcommand{\Q}{{\mathbb Q}}
\newcommand{\R}{{\mathbb R}}
\newcommand{\Z}{{\mathbb Z}}

\newcommand{\abs}[1]{\left\lvert {#1} \right\rvert}

\renewcommand{\leq}{\leqslant}
\renewcommand{\geq}{\geqslant}
\renewcommand{\epsilon}{\varepsilon}
\renewcommand{\phi}{\varphi}

\newcommand{\Ts}{{\mathscr T}}

\DeclareMathOperator{\scl}{scl}
\DeclareMathOperator{\interior}{int}
\DeclareMathOperator{\TP}{TP}
\DeclareMathOperator{\lcm}{lcm}

\newcommand{\sides}{\mathcal S}

\numberwithin{equation}{section}

\makeatletter
  \def\tagform@#1{\maketag@@@{%
   \textbf{(\ignorespaces#1\unskip\@@italiccorr)}}}%
   \renewcommand{\eqref}[1]{\textup{\maketag@@@{(\ignorespaces%
        {\ref{#1}}\unskip\@@italiccorr)}}}
\makeatother

\definecolor{grn}{HTML}{37bc61}
\definecolor{purp}{HTML}{8a2be2}

\begin{document}

\title[Scl of non-filling curves]{On stable commutator length
  of non-filling curves in surfaces}

\author{Max Forester}
\address{Mathematics Department\\
        University of Oklahoma\\
        Norman, OK 73019\\
        USA}
\email{mf@ou.edu}
\author{Justin Malestein}
\address{Mathematics Department\\
        University of Oklahoma\\
        Norman, OK 73019\\
        USA}
\email{justin.malestein@ou.edu}

\maketitle

\begin{abstract}
We give a new proof of rationality of stable commutator length (scl) of
certain elements in surface groups: those represented by curves that
do not fill the surface. Such elements always admit extremal
surfaces for scl. These results also hold more generally for
non-filling $1$--chains. 
\end{abstract}

\section{Introduction}

This paper concerns the computation of stable commutator length (scl) in
fundamental groups of closed orientable surfaces. It is of
considerable interest to know whether scl is always rational in these
groups, or more generally, whether the scl norm is piecewise rational
linear (in the sense of \cite{Calegari1}). 

In more concrete terms, let $\Sigma$ be a closed orientable surface
and $\gamma \co S^1 \to \Sigma$ a null-homologous loop. We are
interested in identifying efficient maps of surfaces 
into $\Sigma$, with boundary mapping to a positive power of
$\gamma$. ``Efficient'' means that the ratio of the topological
complexity of the surface to the power of $\gamma$ on the boundary is
as small as possible. The infimal value of this ratio is the stable
commutator length of $\gamma$. In some cases, there may be a surface
realizing this infimal value; these are called \emph{extremal
surfaces} for scl, and when they exist scl is rational. 

Calegari \cite{Calegari3} has shown that one can always find an
\emph{immersed} surface with boundary a power of $\gamma$, and such
surfaces are the most efficient among those in their relative homology
classes. However there are infinitely many relative classes, and
one cannot draw any immediate conclusions about the existence of
an extremal surface for $\gamma$. 

In the present paper we give a new proof the following theorem of
Calegari, which establishes rationality in the special
case of curves $\gamma$ that do not fill $\Sigma$. It 
appears as Example 4.51 in \cite{sclbook}, as an application of 
Theorem 4.47. We state the result here for integral $1$--chains,
i.e. immersed multicurves. 

\begin{theorem}[Calegari] \label{mainthm}
Let $\Sigma$ be a closed orientable surface of genus $> 1$ and
$\gamma\co \coprod S^1 \to \Sigma$ a null-homologous integral
$1$-chain which does not fill $\Sigma$. Then the stable commutator
length of $\gamma$ is rational and computable. Moreover there exists
an extremal surface for $\gamma$.
\end{theorem}
It follows immediately that \emph{rational} $1$--chains supported on
non-filling multicurves also have rational stable commutator
length. Indeed, we find that scl is piecewise rational linear on
``non-filling'' rational subspaces of the space of $1$--chains; that
is, subspaces spanned by rational $1$--chains $c_1, \dotsc, c_k$
such that the union of their supports is a non-filling multicurve. 

\subsection*{Other rationality results}
The list of groups known to have rational scl, excluding those with
$\scl \equiv 0$, is fairly short. Such groups include:
\begin{itemize}
\item free groups \cite{Calegari1}
\item fundamental groups of non-compact Seifert fibered $3$--manifolds
  \cite{Calegari1} 
\item free products of free abelian groups \cite{Calegari2}
\item free products of cyclic groups \cite{Walker}
\item amalgamated products of free abelian groups \cite{Susse}
\item Baumslag--Solitar groups, alternating elements only \cite{CFL}
\item free products of groups in which $\scl \equiv 0$ \cite{Chen1}
\item generalized Baumslag--Solitar groups \cite{Chen2}
\end{itemize}
Surface groups are conspicuously absent from this list, 
despite being among the simplest and most well-understood one-relator
groups. 

It is worth noting that finitely presented groups do not always
have rational scl; a counterexample was given by Zhuang in
\cite{Zhuang}. If one allows non-finitely presented groups, then in
fact every non-negative real number occurs as scl of some element of
a small cancellation group \cite{Heuer1}. 

\subsection*{Methods} 
The general outline of the argument is similar to the case of free
groups as presented in \cite{BCF}. The stable commutator length of
$\gamma$ is the infimum of $\frac{-\chi(S)}{2n(S)}$ over all
admissible surfaces $S$ mapping to $\Sigma$. First we show that
the infimum can be taken over the much smaller subset $\Ts(\gamma)$ of 
``taut'' admissible surfaces---these will be described below and
defined precisely in Section \ref{mappingsec}. 

Taut surfaces can be encoded as integer vectors via a map 
$v\co \Ts(\gamma) \to \R^k$, and the function
$\frac{-\chi(S)}{2n(S)}$ factors through this map as $\Ts(\gamma)
\overset{v}{\to} \R^k \to \R$ where the second map is a ratio
$A(v)/B(v)$ of linear functions. The computation of $\scl(\gamma)$ is
then reduced to a linear programming problem after showing that the
image of $v$ in $\R^k$ fills (in a suitable sense) the integer points
of a finite-sided rational polyhedron in $\R^k$. 

The two main steps are \emph{encoding}, which entails defining
$\Ts(\gamma)$ and the function $v$; and \emph{reassembly}, in which
specific vectors are shown to be in the image of $v$. Our main
assumption throughout the paper is that $\Sigma$ contains a subsurface
$\Sigma_1$ with non-trivial topology that is disjoint from the image
of $\gamma$. This property is needed for both steps. 

To make a surface $S$ taut, one needs to put it into a normal form
that can be encoded as a finite integer vector. Following Gabai
\cite{Gabai}, we arrange that away from a region containing the double
points of $\gamma$, the preimage in $S$ maps by a branched
immersion. We arrange further that the remaining parts of $S$ fall
into finitely many configurations, and all branch
points lie above $\Sigma_1$. The branch points are now the main
obstactle to having a finite encoding, but these can be removed by
passing to finite-sheeted covers, using the non-trivial topology of
$\Sigma_1$. 

For reassembly we show that every integer vector in the rational
polyhedron has a multiple that encodes a taut surface. When building
this surface out of pieces, a phenomenon involving branch points
arises. A taut surface cannot be built directly, but one can build 
one with branch points. The presence of branch points means that
the linear functional on $\R^k$ reporting $\chi(S)$ does not report
this number correctly. Once again, using finite coverings of
$\Sigma_1$, we are able to eliminate these branch points and construct
a taut surface with the desired encoding.

In contrast, Calegari's approach is quite different. 
He proves Theorem \ref{mainthm} by viewing surface groups as
fundamental groups of graphs of groups in which vertex groups are free
and edge groups are cyclic. If $\gamma$ is non-filling then it avoids
an essential annulus, and one may encode $\Sigma$ as a graph of groups
with one edge. The proof then proceeds by reducing the computation of
$\scl$ to computing $\scl$ in each vertex group, using the assumption
that $\gamma$ is supported in the vertex groups. 

\begin{ack}
The authors were partially supported by Simons Foundation awards
\#638026 (MF) and \#713006 (JM). They thank the referee for making
helpful suggestions. 
\end{ack}

\section{Preliminaries}

\subsection*{Stable commutator length}
We start with basic working definitions for stable commutator length
of group elements and of $1$--chains. See \cite[Chapter 2]{sclbook} for
more information. 

\begin{definition}[scl of integral $1$--chains]
Let $X$ be a space with fundamental group $G$. An \emph{integral
$1$--chain over $G$} is a finite formal sum $c = \sum_i g_i$ with $g_i \in
G$. Let $\gamma_i \co S^1 \to X$ be a loop representing $g_i$ for each
$i$ and let $\gamma \co \coprod_i S^1 \to X$ be the map given by
$\coprod_i \gamma_i$. An \emph{admissible surface for $\gamma$} (or
\emph{for $c$}) is a
compact oriented surface $S$ together with a map $f \co S \to X$ such
that 
\begin{itemize}
\item $\partial S \not= \emptyset$ and $S$ has no sphere or disk components
\item the restriction $f \vert_{\partial S}$ factors through $\gamma$;
that is, there is a commutative diagram
\[
\begin{tikzcd}
\partial S \arrow[hook]{r} \arrow{d} & S \arrow{d}{f} \\
\coprod_i S^1 \arrow{r}{\coprod_i \gamma_i \ } & X
\end{tikzcd}
\]
\item the restriction of the map $\partial S \to \coprod_i S^1$ to each
connected  component of $\partial S$ is a map of \emph{positive}
degree
\item there is a positive integer $n(S)$ such that for each component
  of $\coprod_i S^1$, the preimage in $\partial S$ maps to it by total
  degree $n(S)$. 
\end{itemize}
The \emph{stable commutator length (scl) of $c$} is
defined by
\[ \scl(c) \ = \ \inf_{S} \frac{-\chi(S)}{2n(S)}\]
where the infimum is taken over all admissible surfaces for
$\gamma$. If no admissible surface exists, we define $\scl(c) =
\infty$. 

Note: the third bullet is not included in Calegari's definition of
admissible surface in \cite{sclbook}. However, Proposition 2.13 of
\cite{sclbook} shows that including it does not change the meaning of
scl. 

Using homogeneity properties of $\scl$, the definition extends to
rational chains. If $c = \sum_i c_i g_i$ with $c_i \in \Q$ then $mc$
is integral for some $m$, and we may define $\scl(c) =
\frac{1}{m}\scl(mc)$. Extending by continuity, it is also defined for
real chains. See \cite[Section 2.6]{sclbook}. 
\end{definition}

\begin{definition}[scl of group elements] 
For any $g\in G$ we define $\scl(g)$ by considering $g$ as an integral
$1$--chain consisting of one element. The above definition, in this
case, specializes to the standard definition of $\scl(g)$.
Note that admissible surfaces for $g$ exist if and only if $g^k \in
[G,G]$ for some $k\not= 0$, and so $\scl(g)$ is finite exactly when
the latter occurs. 
\end{definition}

\begin{definition}
Let $c$ be an integral $1$--chain over $G$. An \emph{extremal surface
for $c$} is an admissible surface $S$ for $c$ that realizes the
infimum in the definition of $\scl(c)$. If $c$
is a rational $1$--chain, an extremal surface for $c$ is just an
extremal surface for any $mc$ that is integral. Extremal surfaces need not
exist, but when they do, $\scl(c)$ must of course be rational. 
\end{definition}

Let $C_1(G)$ be the space of $1$--chains, which is a
vector space over $\R$ with basis $G$. Let $B_1(G) \subset C_1(G)$ be
the subspace of boundaries of $2$--chains, i.e. the kernel of the
quotient map $C_1(G) \to H_1(G;\R)$. Let $B_1^H(G)$ be the quotient of
$B_1(G)$ by the subspace $H$ spanned by elements of the form $g^n -
ng$ and $g - hgh^{-1}$ for all $g,h \in G$, $n \in \Z$. Then the
function $\scl$ is a pseudo-norm on $B_1^H(G)$. Moreover, whenever $G$
is hyperbolic, $\scl$ is a geniune norm on $B_1^H(G)$
\cite{CF}.

\subsection*{Surfaces}
Let $\Sigma$ be a closed surface. We say that a multicurve $\gamma \co
\coprod S^1 \to \Sigma$ is \emph{non-filling} if it is homotopic to an
immersion in general position such that some component of the
complement of the image is not a disk (or equivalently, $\Sigma$
contains an essential simple closed curve disjoint from the image of
$\gamma$). Correspondingly, an integral $1$--chain or group element in
$G = \pi_1(\Sigma)$ is \emph{non-filling} if it is represented by a
non-filling multicurve. 

\begin{definition}
A map $f \co S \to \Sigma$ of orientable surfaces is \emph{compressible}
if there is an essential simple closed curve $C \subset S$ such that
$f(C)$ is nullhomotopic in $\Sigma$. If $f$ is not compressible it is
called \emph{incompressible}. 

When $f$ is compressible, one can replace an annular neighborhood of
$C$ by two disks and extend $f$ over these disks, thus obtaining $f'
\co S' \to \Sigma$ of smaller complexity.
\end{definition}

\subsection*{Coverings} 
Throughout this paper curves and surfaces typically have
orientations, and by the \emph{degree} of a map we always mean the
homological degree. 

Suppose $p \co X \to Y$ is a finite-sheeted covering of oriented
manifolds. Partition the connected components of $X$ as $X^+
\sqcup X^-$ such that $p$ is orientation-preserving on $X^+$ and
orientation-reversing on $X^-$. We define the \emph{positive degree}
of $p$ to be the number of sheets of $X^+$ and the \emph{negative
degree} of $p$ to be the number of sheets of $X^-$. Thus the sum of
these numbers is the total number of sheets of $p$ and their
difference is the homological degree. 

Next we record some basic covering lemmas for surfaces. 

\begin{lemma}\label{coverlem} 
Let $S$ be a compact connected oriented surface having $p\geq 1$ boundary
components, with $S \not\cong D^2$. 
\begin{enumerate}
\item \label{cover1} For any $q<p$ boundary components $C_1, \dotsc,
C_{q}$ of $S$ and integers $n_1, \dotsc, n_{q} \geq 1$, there is a connected
regular finite-sheeted cover $S' \to S$ such that every component of
$\partial  S'$ mapping to $C_i$ covers with degree $n_i$. 
\item \label{cover2} For any integer $n \geq 1$ there is a connected
finite-sheeted cover $S' \to S$ such that every component of $\partial
S'$ covers its image with degree $n$. 
\end{enumerate}
\end{lemma}

\begin{proof}
For \eqref{cover1}, let $x_i \in H_1(S)$ be the element represented by
$C_i$, for $1 \leq i \leq q$. This set extends to a basis $\{x_1,
\dotsc, x_n\}$ of $H_1(S)$, because $q<p$. Let $N$ be the least common
multiple of the numbers $n_i$. Let $\phi\co H_1(S) \to \Z/N\Z$ send
$x_i$ to $N/n_i$ for $1 \leq i \leq q$ (the values on the other basis
elements are irrelevant). Now $S'\to S$ is the cover corresponding to
the kernel of the composition $\pi_1(S) \to H_1(S) \to \Z/N\Z$. Since
$\phi(x_i)$ has order $n_i$ in $\Z/N\Z$ and the cover is regular, the
boundary $\partial S'$ behaves as desired.

For \eqref{cover2}, suppose first that $p > 1$. Choose boundary
components $C_1, \dotsc, C_{p-1}$ of $\partial S$ and a basis
$\{x_i\}$ of $H_1(S)$ as before. Let $\phi \co H_1(S) \to
(\Z/n\Z)^{p-1}$ send $x_i$ to the $i$th generator $(0, \dotsc, 1,
\dotsc, 0)$ for $1 \leq i \leq p-1$. Then $\phi$ sends the class
represented by the last boundary component of $S$ to $(-1, \dotsc,
-1)$. Since every boundary component maps to an element of order $n$,
the cover $S' \to S$ corresponding to the kernel of $\pi_1(S) \to
H_1(S) \to (\Z/n\Z)^{p-1}$ has the desired property. 

If $\partial S$ has one component, first let $S'' \to S$ be any
connected $2$--sheeted cover (which exists since $S \not\cong
D^2$). It will have $2$ boundary components, each mapping by degree
$1$, because $\chi(S'')$ is even. Apply the preceding case to get
$S' \to S''$, and take the composition $S' \to S'' \to S$. 
\end{proof}

It will be convenient to have a version of this lemma for 
disconnected surfaces. By a finite-sheeted cover of a disconnected
surface, we mean a cover for which the number of sheets is the same
over every component.  

\begin{lemma}\label{coverlem2} 
Let $S$ be a compact oriented surface where each component has at
least $1$ boundary component and no component is homeomorphic to
$D^2$.  For any set of boundary components $C_1, \dotsc, C_{q}$ of $S$
which does not contain all the boundary components of any component of
$S$ and integers $n_1, \dotsc, n_{q} \geq 1$, there is a
finite-sheeted cover $S' \to S$ such that every component of $\partial
S'$ mapping to $C_i$ covers with degree $n_i$.
\end{lemma}
\begin{proof}
Let $S_1, \dots, S_m$ be the components of $S$. We first apply Lemma
\ref{coverlem} \eqref{cover1} to each $S_i$ to obtain an $N_i$-fold
cover $S_i' \to S_i$ with the correct degree on boundary
components. Let $N = \lcm(N_1, \dots, N_m)$. Then the cover of $S$
consisting of $N/N_i$ copies of $S_i$ has the desired properties. 
\end{proof}

We are also interested in the case where we allow branched covers with
prescribed degrees on the boundaries. The flexibility of degrees on
boundaries is much larger for branched covers. 
\begin{lemma}\label{coverlem3}
Let $S$ be a compact connected oriented surface having $p\geq 1$
boundary components. Let $C_1, \dots, C_p$ be the boundary components
of $S$, let $n$ be a natural number, and let $\lambda_1, \dots,
\lambda_p$ be partitions of $n$. Then there is an $n$-fold branched
cover $S' \to S$ such that the degrees of the components over $C_i$
are precisely the values in the partition $\lambda_i$. 
\end{lemma}
Note that $S'$ is not necessarily connected. 
\begin{proof}
We start with the cover $S'' \to S$ where $S''$ is $n$ copies of
$S$. Let $A, B$ be distinct copies of $C_i$ in $S''$. Let $\alpha,
\beta$ be copies of the same small arc from a point on $C_i$ into the
interior of $S$. Cut $S''$ along $\alpha$ and $\beta$ and glue the
right side of $\beta$ to the left side of $\alpha$ and vice
versa. This new surface, in a canonical way, admits a branched cover
to $S$ where the degrees over $C_i$ are $2, 1, 1, \dots, 1$. By using
this kind of surgery to connect boundary components of $S''$, it is
clear that we can construct a branched cover $S' \to S$ where the
degrees over $C_i$ are an arbitrary partition of $n$. 
\end{proof}

\section{Mapping surfaces to surfaces}\label{mappingsec}

\subsection*{Setup}
Let $\Sigma$ be a closed oriented surface of genus greater than one. 
Let $\gamma \co \coprod S^1 \to \Sigma$ be a null-homologous
$1$--chain that does not fill $\Sigma$. We can put $\gamma$ into
general position (smoothly immersed, transverse intersections, no
triple points) such that at least one complementary component is not a
disk. 

Let $\Sigma_0$ be one such complementary component.
Let $B\subset \Sigma_0$ be an embedded closed disk and let $\lambda_1,
\dotsc, \lambda_k$ be disjoint properly embedded arcs in $\Sigma_0 -
\interior(B)$ with endpoints on $\partial B$ such that the inclusion
of $B \cup \bigcup_i \lambda_i$ 
into $\Sigma_0$ is a homotopy equivalence. Let $L \subset (\Sigma_0 -
\interior(B))$ be the closure of a tubular neighborhood of $\bigcup_i
\lambda_i$, so that $\Sigma_1 = B \cup L$ is compact subsurface
homotopy equivalent to $\Sigma_0$. 

Next let $\mu_1, \dotsc, \mu_{\ell}$ be disjoint properly embedded
arcs in $\Sigma - \interior(\Sigma_1)$ with endpoints on
$\partial B - L$ such that the $\mu_i$ cut each component
of $\Sigma - \Sigma_1$ into a single disk and the $\mu_i$ do not cross
any double points of $\gamma$. We also require
that the arcs $\mu_i$ are transverse to the immersed submanifold
$\gamma(\coprod S^1)$. Let $M \subset
(\Sigma - \interior(\Sigma_1))$ be the closure of a tubular 
neighborhood of $\bigcup_i \mu_i$, disjoint from $L$, so that
$\Sigma_2 = \Sigma_1 \cup M$ contains no double point of $\gamma$
and $\gamma(\coprod S^1) \cap M$ consists of fibers of $M$. 
Let $D$ be the
closure of $\Sigma - \Sigma_2$; this is a disjoint union of
embedded closed disks. Finally, define $\Sigma_3 = \Sigma -
\interior(\Sigma_1) = D \cup M$. 

Each component of $\coprod S^1$ maps by $\gamma$ to a path which
alternates between embedded oriented arcs in $D$, called \emph{turn
arcs}, and arcs crossing $M$. Let $\tau_1, \dotsc, \tau_m$ be the
turn arcs. See Figure \ref{sigmapic}. 

\begin{figure}[!ht]
\begin{tikzpicture}[line width = 0.3mm, scale = 0.8]

\begin{scope}[xscale=-1]
\draw (-6.7, 1) circle (1.2cm);
\draw (-6.7, 1) circle (0.8cm);
\fill[white, even odd rule] (-6, 1.2) circle[radius=1.2cm] circle[radius=0.8cm];
\draw (-6, 1.2) circle (1.2cm);
\draw (-6, 1.2) circle (0.8cm);
\draw (0, 2) circle (1.2cm);
\draw (0, 2) circle (0.8cm);
\fill[gray!20!white, even odd rule] (1, 2) circle[radius=1.2cm] circle[radius=0.8cm];
\draw[red] (1, 2) circle (1.2cm);
\draw[red] (1, 2) circle (0.8cm);
\draw (6, 1.2) circle (1.2cm);
\draw (6, 1.2) circle (0.8cm);
\fill[white, even odd rule] (6.7, 1) circle[radius=1.2cm] circle[radius=0.8cm];
\draw (6.7, 1) circle (1.2cm);
\draw (6.7, 1) circle (0.8cm);
\fill[gray!20!white, even odd rule] (-6, 0) circle[radius=3.4cm] circle[radius=3cm];
\draw[red] (-6, 0) circle (3.4cm);
\draw[red] (-6, 0) circle (3cm);

\filldraw[fill=gray!20!white, draw = red] (0, 0) ellipse (8cm and 2cm);

\begin{scope}[decoration = {markings, mark = at position 0.6 with  {\arrow{stealth}}}]
\draw[red, postaction={decorate}] (-45:8cm and 2cm) arc(-45:-35:8cm and 2cm); 
\draw[red, shift={(-6,0)}, postaction={decorate}] (210:3cm) arc(210:202:3cm); 
\draw[red, shift={(1,2)}, postaction={decorate}] (80:1.2cm) arc(80:120:1.2cm); 
\draw[red, shift={(1,2)}, postaction={decorate}] (70:0.8cm) arc(70:50:0.8cm); 
\end{scope}

\draw[black, line width= 0.32mm] (243.4:8cm and 2cm) arc(243.4:246.6:8cm and 2cm);
\draw[black, line width= 0.32mm] (158.3:8cm and 2cm) arc(158.3:166.7:8cm and 2cm);
\draw[black, line width= 0.32mm] (147.8:8cm and 2cm) arc(147.8:153.7:8cm and 2cm);
\draw[black, line width= 0.32mm] (134.1:8cm and 2cm) arc(134.1:138.5:8cm and 2cm);
\draw[black, line width= 0.32mm] (127.3:8cm and 2cm) arc(127.3:131.32:8cm and 2cm);
\draw[black, line width= 0.32mm] (113.4:8cm and 2cm) arc(113.4:116.6:8cm and 2cm);
\draw[black, line width= 0.32mm] (95.62:8cm and 2cm) arc(95.62:98.75:8cm and 2cm);
\draw[black, line width= 0.32mm] (88.7:8cm and 2cm) arc(88.7:91.3:8cm and 2cm);
\draw[black, line width= 0.32mm] (81.24:8cm and 2cm) arc(81.24:84.37:8cm and 2cm);
\draw[black, line width= 0.32mm] (74.2:8cm and 2cm) arc(74.2:76.85:8cm and 2cm);
\draw[black, line width= 0.32mm] (48.68:8cm and 2cm) arc(48.68:52.7:8cm and 2cm);
\draw[black, line width= 0.32mm] (41.5:8cm and 2cm) arc(41.5:45.9:8cm and 2cm);
\draw[black, line width= 0.32mm] (26.3:8cm and 2cm) arc(26.3:32.2:8cm and 2cm);
\draw[black, line width= 0.32mm] (13.3:8cm and 2cm) arc(13.3:21.7:8cm and 2cm);

\draw (0, 0) node{$B$};
\draw (-45:8.3cm and 2.3cm) node{\Small $\alpha_1$ \normalsize};
\draw (46:7.7cm and 1.7cm) node{\tiny $\alpha_2$ \normalsize};
\draw (35:7.7cm and 1.7cm) node{\tiny $\alpha_3$ \normalsize};
\draw (21.5:7.7cm and 1.7cm) node{\tiny $\alpha_4$ \normalsize};
\draw[shift = {(1,2)}] (85:1.45cm) node{\Small $\alpha_5$ \normalsize};
\draw[shift = {(1,2)}] (40:0.55cm) node{\tiny $\alpha_6$ \normalsize};
\draw[shift = {(-6,0)}] (215:2.7cm) node{\Small $\alpha_7$ \normalsize};
\draw (158:7.7cm and 1.7cm) node{\tiny $\alpha_8$ \normalsize};
\draw (144:7.7cm and 1.7cm) node{\tiny $\alpha_9$ \normalsize};
\draw (133:7.7cm and 1.7cm) node{\tiny $\alpha_{10}$ \normalsize};

\begin{scope}[decoration = {markings, mark = at position 0.8 with  {\arrow{stealth}}}]
\draw[blue, shift = {(-6.7, 1)}, postaction={decorate}] (170:0.8cm) -- (170:1.2cm);
\draw[blue, shift = {(-6.7, 1)}, postaction={decorate}] (150:1.2cm) -- (150:0.8cm);
\draw[blue, shift = {(-6, 1.2)}, postaction={decorate}] (122:0.8cm) -- (122:1.2cm);
\draw[blue, shift = {(-6, 1.2)}, postaction={decorate}] (102:1.2cm) -- (102:0.8cm);
\draw[blue, shift = {(0, 2)}, postaction={decorate}] (160:1.2cm) -- (160:0.8cm);
\draw[blue, shift = {(0, 2)}, postaction={decorate}] (140:0.8cm) -- (140:1.2cm);
\draw[blue, shift = {(0, 2)}, postaction={decorate}] (120:1.2cm) -- (120:0.8cm);
\draw[blue, shift = {(0, 2)}, postaction={decorate}] (100:0.8cm) -- (100:1.2cm);
\draw[blue, shift = {(6, 1.2)}, postaction={decorate}] (120:1.2cm) -- (120:0.8cm);
\draw[blue, shift = {(6, 1.2)}, postaction={decorate}] (100:0.8cm) -- (100:1.2cm);
\draw[blue, shift = {(6.7, 1)}, postaction={decorate}] (85:0.8cm) -- (85:1.2cm);
\draw[blue, shift = {(6.7, 1)}, postaction={decorate}] (65:1.2cm) -- (65:0.8cm);
\end{scope}

\end{scope}

\draw (-4, -7) circle (2.5cm);
\draw (4, -7) circle (2.5cm);
\draw (-4, -10.5) node{$D_1$};
\draw (4, -10.5) node{$D_2$};

\begin{scope}[decoration = {markings, mark = at position 0.6 with  {\arrow{stealth}}}]
\draw[shift={(-4,-7)}, red, line width = 0.5mm, postaction={decorate}] (-15:2.5cm) arc(-15:15:2.5cm) node[anchor = west, pos = 0.5, black]{$\alpha_1$};
\draw[shift={(-4,-7)}, red, line width = 0.5mm, postaction={decorate}] (45:2.5cm) arc(45:75:2.5cm) node[anchor = 240, pos = 0.5, black]{$\alpha_2$};
\draw[shift={(-4,-7)}, red, line width = 0.5mm, postaction={decorate}] (105:2.5cm) arc(105:135:2.5cm)  node[anchor = 300, pos = 0.5, black]{$\alpha_3$};
\draw[shift={(-4,-7)}, red, line width = 0.5mm, postaction={decorate}] (165:2.5cm) arc(165:195:2.5cm) node[anchor = east, pos = 0.5, black]{$\alpha_4$};
\draw[shift={(-4,-7)}, red, line width = 0.5mm, postaction={decorate}] (225:2.5cm) arc(225:255:2.5cm) node[anchor = 60, pos = 0.5, black]{$\alpha_5$};
\draw[shift={(-4,-7)}, red, line width = 0.5mm, postaction={decorate}] (285:2.5cm) arc(285:315:2.5cm) node[anchor = 120, pos = 0.5, black]{$\alpha_6$};

\draw[shift={(4,-7)}, red, line width = 0.5mm, postaction={decorate}] (-22.5:2.5cm) arc(-22.5:22.5:2.5cm) node[anchor = 180, pos = 0.5, black]{$\alpha_7$};
\draw[shift={(4,-7)}, red, line width = 0.5mm, postaction={decorate}] (67.5:2.5cm) arc(67.5:112.5:2.5cm) node[anchor = 270, pos = 0.5, black]{$\alpha_ 8$};
\draw[shift={(4,-7)}, red, line width = 0.5mm, postaction={decorate}] (157.5:2.5cm) arc(157.5:202.5:2.5cm) node[anchor = 0, pos = 0.5, black]{$\alpha_9$};
\draw[shift={(4,-7)}, red, line width = 0.5mm, postaction={decorate}] (247.5:2.5cm)
arc(247.5:292.5:2.5cm) node[anchor = 90, pos = 0.5,
black]{$\alpha_{10}$};
\end{scope}

\begin{scope}[decoration = {markings, mark = at position 0.15 with  {\arrow{stealth}}}]
\draw[shift={(-4, -7)}, blue, line width = 0.4mm, postaction={decorate}] (321:2.5cm) to[out=140, in = 275] (95:2.5cm);
\draw[shift={(-4, -7)}, blue, line width = 0.4mm, postaction={decorate}] (205:2.5cm) to[out=385, in = 507] (327:2.5cm);
\draw[shift={(-4, -7)}, blue, line width = 0.4mm, postaction={decorate}] (333:2.5cm) to[out=513, in = 325] (145:2.5cm);
\draw[shift={(-4, -7)}, blue, line width = 0.4mm, postaction={decorate}] (85:2.5cm) to[out=265, in = 519] (339:2.5cm);
\draw[shift={(-4, -7)}, blue, line width = 0.4mm, postaction={decorate}] (261:2.5cm) to[out=441, in = 205] (25:2.5cm);
\draw[shift={(-4, -7)}, blue, line width = 0.4mm, postaction={decorate}] (35:2.5cm) to[out=215, in = 459] (279:2.5cm);
\draw[shift={(-4, -7)}, blue, line width = 0.4mm, postaction={decorate}] (155:2.5cm) to[out=335, in = 447] (267:2.5cm);
\draw[shift={(-4, -7)}, blue, line width = 0.4mm, postaction={decorate}] (273:2.5cm) to[out=453, in = 395] (215:2.5cm);

\draw[shift={(4, -7)}, blue, line width = 0.4mm, postaction={decorate}] (232.5:2.5cm) to[out=412.5, in = 217.5] (37.5:2.5cm);
\draw[shift={(4, -7)}, blue, line width = 0.4mm, postaction={decorate}] (52.5:2.5cm) to[out=232.5, in = 307.5] (127.5:2.5cm);
\draw[shift={(4, -7)}, blue, line width = 0.4mm, postaction={decorate}] (142.5:2.5cm) to[out=322.5, in = 487.5] (307.5:2.5cm);
\draw[shift={(4, -7)}, blue, line width = 0.4mm, postaction={decorate}] (322.5:2.5cm) to[out=502.5, in = 397.5] (217.5:2.5cm);
\end{scope}
\end{tikzpicture}

\caption{A surface $\Sigma$ and its handle decomposition $\Sigma = B
  \cup (L \cup M) \cup D$. The shaded part is the subsurface $\Sigma_1
  = B \cup L$. The null-homologous $1$--cycle $\gamma$ is shown in
  blue. All self-intersections of $\gamma$ lie inside $D = D_1 \cup
  D_2$. Outside of $D$, $\gamma$ meets only the $1$--handles $M$ and
  crosses them transversely. The arcs of $\gamma$ crossing $D$ are
  called turn arcs. The arcs $\alpha_i$ are the maximal sub-arcs of
  $\partial D$ that meet $\partial \Sigma_1$. }\label{sigmapic}
\end{figure}
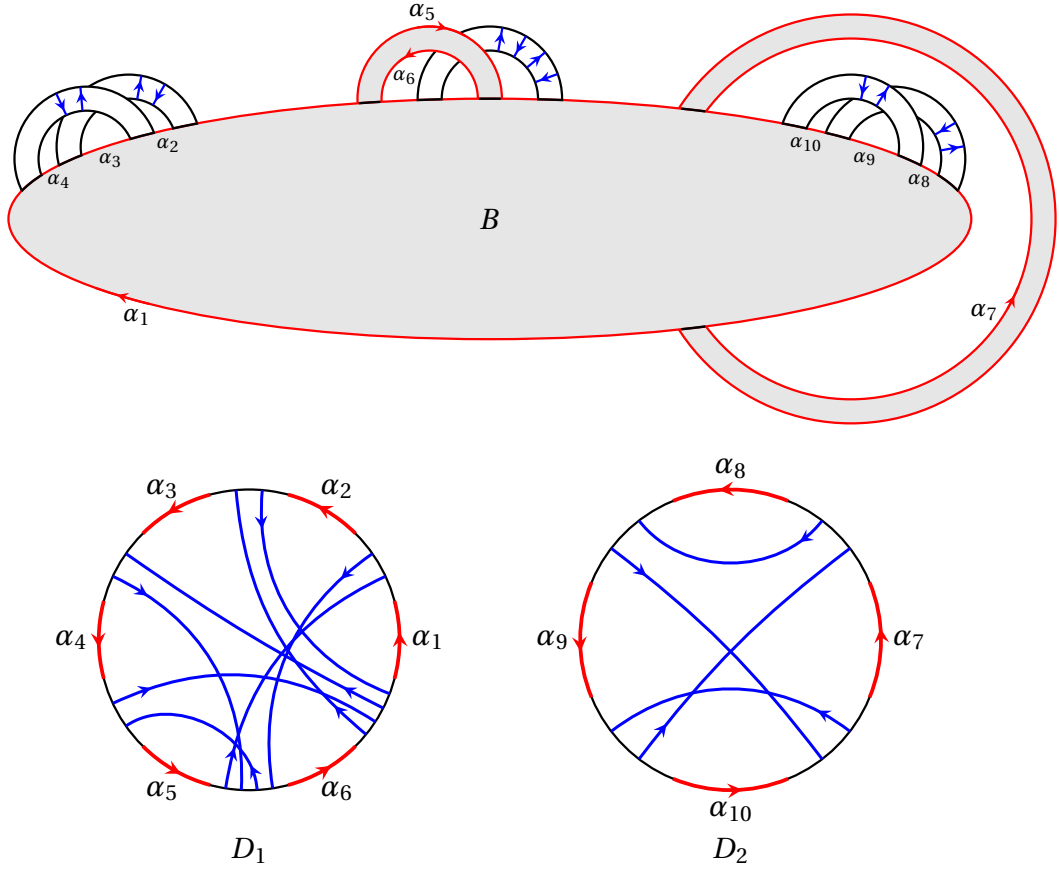

\begin{definition}[Turn paths] 
A \emph{turn path} is a loop $S^1 \to D$ which is a concatenation of
paths, alternating between turn arcs and immersed paths
in $\partial D$. Immersions $S^1 \to \partial D$ are included. 
We consider two turn paths to be the same if they differ by an
orientation-preserving reparametrization of $S^1$. See Figure
\ref{turnpic}. 

Each component of the oriented boundary of $D$ is a concatenation of
arcs, alternating between arcs in the boundary of $M$ and arcs in the
boundary of $\Sigma_1$. Let $\alpha_1, \dotsc, \alpha_n$ be the arcs
of $\partial D$ which lie in the boundary of $\Sigma_1$. 

A turn path is \emph{taut} if it uses each turn arc at most once and
each arc from $\{\alpha_i\} \cup \{\overline{\alpha}_i\}$ at most once
(where $\overline{\alpha}_i$ is the reverse of $\alpha_i$). Note that
there are only finitely many taut turn paths.  
\end{definition}

\begin{figure}[!ht]
\begin{tikzpicture}[line width = 0.3mm, scale = 0.8]

\draw (4, 0) circle (2.5cm);
\draw[-stealth] (-2, 0) -- (-.5, 0);

\small
\begin{scope}[decoration = {markings, mark = at position 0.6 with  {\arrow{stealth}}}]
\draw[shift={(4,0)}, red, line width = 0.4mm, postaction={decorate}] (45:2.5cm) arc(45:75:2.5cm) node[anchor = 240, pos = 0.5, black]{$\alpha_2$};
\draw[shift={(4,0)}, red, line width = 0.4mm, postaction={decorate}] (285:2.5cm) arc(285:315:2.5cm) node[anchor = 120, pos = 0.5, black]{$\alpha_6$};
\end{scope}

\begin{scope}[decoration = {markings, mark = at position 0.15 with  {\arrow{stealth}}}]
\draw[shift={(4, 0)}, lightgray, line width = 0.4mm, postaction={decorate}] (321:2.5cm) to[out=140, in = 275] (95:2.5cm);
\draw[shift={(4, 0)}, lightgray, line width = 0.4mm, postaction={decorate}] (205:2.5cm) to[out=385, in = 507] (327:2.5cm);
\draw[shift={(4, 0)}, lightgray, line width = 0.4mm, postaction={decorate}] (333:2.5cm) to[out=513, in = 325] (145:2.5cm);
\draw[shift={(4, 0)}, lightgray, line width = 0.4mm, postaction={decorate}] (35:2.5cm) to[out=215, in = 459] (279:2.5cm);
\draw[shift={(4, 0)}, lightgray, line width = 0.4mm, postaction={decorate}] (155:2.5cm) to[out=335, in = 447] (267:2.5cm);
\draw[shift={(4, 0)}, lightgray, line width = 0.4mm, postaction={decorate}] (273:2.5cm) to[out=453, in = 395] (215:2.5cm);
\draw[shift={(4, 0)}, blue, line width = 0.4mm, postaction={decorate}] (85:2.5cm) to[out=265, in = 519] (339:2.5cm);
\draw[shift={(4, 0)}, blue, line width = 0.4mm, postaction={decorate}] (261:2.5cm) to[out=441, in = 205] (25:2.5cm);
\end{scope}

\begin{scope}[cap=round, decoration = {markings, mark = at position 0.35 with  {\arrow{stealth}}}]
\draw[shift={(-6, 0)}, mygreen, line width = 0.5mm] (85:2.5cm) to[out=265, in = 519] (339:2.5cm);
\draw[shift={(-6, 0)}, white, line width = 3mm, postaction={decorate}] (261:2.5cm) to[out=441, in = 205] (25:2.5cm);
\draw[shift={(-6, 0)}, mygreen, line width = 0.5mm, postaction={decorate}] (261:2.5cm) to[out=441, in = 205] (25:2.5cm);
\end{scope}

\begin{scope}[decoration = {markings, mark = at position 0.5 with  {\arrow{stealth}}}]
\draw[shift={(-6, 0)}, mygreen, line width = 0.5mm] (25:2.5cm) arc(25:85:2.5cm);
\draw[shift={(-6, 0)}, mygreen, line width = 0.5mm] (339:2.5cm) arc(339:261:2.5cm);
\end{scope}

\draw (4.6,1.75) node{$\tau_1$};
\draw (3.25,-1.9) node{$\tau_2$};
\end{tikzpicture}

\caption{A taut turn path in $D_1$ which runs over $\tau_1$,
  $\overline{\alpha}_6$, $\tau_2$, and $\alpha_2$. 
Every turn path extends to a map of a disk; in this case the disk is
twisted. 
} \label{turnpic} 
\end{figure}
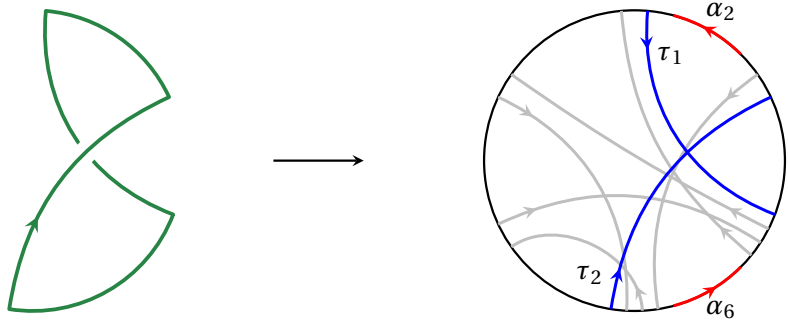

\begin{definition}\label{tautdef}
An admissible map $f\co S \to \Sigma$ is \textit{taut} if 
\begin{enumerate}
\item \label{t1} $f \co f^{-1}(\Sigma_2) \to \Sigma_2$ is an immersion
\item \label{t2} $f^{-1}(D)$ is a disjoint union of disks,
\item \label{t3} for each component $E$ of $f^{-1}(D)$, the boundary map
$f|_{\partial E}$ is a taut turn path. 
\end{enumerate}
The set of taut admissible maps for $\gamma$ will be denoted by
$\Ts(\gamma)$. The disks in \eqref{t2} and \eqref{t3} will be called
\emph{turn disks}.
\end{definition}

The main result of this section is the following.

\begin{proposition}\label{tautprop}
Suppose $f\co S \to \Sigma$ is an admissible map. Then there
exists a taut admissible map $f'\co S' \to \Sigma$ such that 
$\frac{-\chi(S')}{2n(S')} \leq \frac{-\chi(S)}{2n(S)}$.
\end{proposition}

\begin{corollary}\label{tautcor}
There is an equality 
\[\scl(\gamma) \ = \ \inf_{S \in \Ts(\gamma)}
  \frac{-\chi(S)}{2n(S)}.\] 
\end{corollary}

The next lemma provides the first step toward
Proposition~\ref{tautprop}. It achieves all the desired properties 
except for tautness of the turn paths in item \eqref{t3}. 

\begin{lemma}\label{tautlemma1}
Suppose $f\co S \to \Sigma$ is an admissible map which is
incompressible. Then $f$ is homotopic to an admissible map $g\co
S \to \Sigma$ such that 
\begin{enumerate}
\item \label{tt1} $g \co g^{-1}(\Sigma_2) \to \Sigma_2$ is an immersion
\item \label{tt2} $g^{-1}(D)$ is a disjoint union of disks,
\item \label{tt3} for each component $E$ of $g^{-1}(D)$, the boundary map
$g|_{\partial E}$ is a turn path. 
\end{enumerate}
The boundary maps $f\vert_{\partial S}$ and $g\vert_{\partial S}$
differ only by reparametrization. 
\end{lemma}

\begin{proof}
We proceed as in Step 1 of \cite[Theorem 2.1]{Gabai}. Homotope $f$
rel boundary to $f_1$ which is transverse to the disk $B$. We then
have that each component of $f_1^{-1}(B)$ is a disk 
mapping homeomorphically onto $B$. Next, let $T = S -
\interior(f_1^{-1}(B))$. Since $f_1\vert_{\partial T}$ is
transverse to the arcs $\lambda_i$ and $\mu_i$, we may homotope $f_1$
rel $\partial S \cup f_1^{-1}(B)$ to be transverse to these
arcs. Among all maps $f_2$ homotopic to $f$ rel $\partial S$ that are
transverse to $B$ and transverse to the arcs $\lambda_i$ and $\mu_i$
on $S - \interior(f_2^{-1}(B))$,
let $f_2$ minimize the \emph{complexity}, which is the pair 
$\left(\abs{\pi_0\bigl(f_2^{-1}\bigl(B\bigr)\bigr) }, \abs{
\pi_0\bigl(f_2^{-1}\bigl(\bigcup_i \lambda_i \cup \bigcup_i
\mu_i\bigr)\bigr)}\right)$ ordered lexicographically. 

Since $f_2$ is incompressible, no component of $f_2^{-1}(\lambda_i)$
or $f_2^{-1}(\mu_i)$ is an essential closed curve. No component
is an inessential closed curve either. Suppose $C$ is such a curve
mapping to $\lambda_i$, say. For a sufficiently close arc $\lambda_i'$
that is parallel to $\lambda_i$, there is a curve $C'$ parallel to $C$
mapping to $\lambda_i'$ that bounds a disk $E \subset S$, with $C$ in
the interior of $E$. By re-defining the map on $\interior(E)$ to map
into $\lambda_i'$ one obtains a map of smaller complexity; the new map
is homotopic to $f_2$ rel $\partial S$ because $\pi_2(\Sigma)=0$.  

Thus each component of $f_2^{-1}(\lambda_i)$ is an arc with endpoints
mapping to $\partial \lambda_i$, and each component of
$f_2^{-1}(\mu_i)$ is an arc with endpoints mapping to $\partial \mu_i
\cup \{\mu_i \cap \gamma(\coprod S^1)\}$. Note that in the second
case, both endpoints cannot map to the same point $z \in \{\mu_i \cap
\gamma(\coprod S^1)\}$, by orientation considerations. There is a
consistent transverse orientation on $f_2^{-1}(\mu_i)$, but
$f_2\vert_{\partial S}$ always passes through $z$ in the same
direction, being a positive power of $\gamma$. 

Since complexity is minimized, both endpoints cannot map to the same
point of $\partial \mu_i$; otherwise,
$\abs{\pi_0\bigl(f_2^{-1}\bigl(\bigcup_i \lambda_i \cup \bigcup_i
\mu_i\bigr)\bigr)}$ could be reduced, and then
$\abs{\pi_0\bigl(f_2^{-1}\bigl(B\bigr)\bigr)}$ as
well by homotopy rel $\partial S$ to a new map which is still transverse
to $B \cup \{\mu_i\} \cup \{\lambda_i\}$. Hence the endpoints map to
distinct points. The same reasoning applies to components of
$f_2^{-1}(\lambda_i)$. Hence the endpoints of every arc in
$f_2^{-1}(\lambda_i)$ or $f_2^{-1}(\mu_i)$ map to distinct points. 

Now, by a further homotopy rel $\partial S \cup f_2^{-1}(B)$, the
components of $f_2^{-1}(\lambda_i)$ and $f_2^{-1}(\mu_i)$ can be
tightened so that each maps by an immersion; let $f_3$ be the
resulting map. There are compatible tubular neighborhoods on which
$f_3$ is a bundle map, and by another homotopy to $f_4$ (which expands
along fibers of the tubular neighborhoods $L$ and $M$) we can ensure
that each component of $f_4^{-1}(L)$ and $f_4^{-1}(M)$ maps by an
immersion. Now $f_4 \co f_4^{-1}(\Sigma_2) \to \Sigma_2$ is an
immersion, and $f_4$ satisfies \eqref{tt1}. 

Next consider a component $E$ of $f_4^{-1}(D)$. Each boundary
component of $\partial E$ maps by a concatenation of paths that
alternate between turn arcs and paths in $\partial D$;
moreover these latter paths are part of the boundary map of $f_4 \co
f_4^{-1}(\Sigma_2) \to \Sigma_2$ and hence are immersions. Thus
$\partial E$ maps to $D$ by turn paths and $f_4$ satisfies \eqref{tt3}. 

Next, suppose some component $E$ of $f_4^{-1}(D)$ is not a disk. Then
there is a simple closed curve $C \subset \interior(E)$ which is
essential in $E$. Since $f_4(C)$ is nullhomotopic in $\Sigma$, we have
that $C$ bounds a disk $E' \subset S$. Since $\pi_2(\Sigma)=0$ there
is a homotopy rel $S - \interior(E')$ from $f_4$ to $f_5$ such that
$f_5(E') \subset D$. By performing finitely many homotopies of this
type, we arrive at $g$ such that each component of $g^{-1}(D)$ is a
disk and $g$ agrees with $f_4$ on $g^{-1}(\Sigma - \interior(D)) =
g^{-1}(\Sigma_2)$. In particular, $g$ satisfies \eqref{tt1} and
\eqref{tt2}. 

Lastly, $f_4^{-1}(\partial D) \supseteq g^{-1}(\partial D)$ and $f_4$
and $g$ agree on $g^{-1}(\partial D)$ (and on $\partial S$), so the
boundary paths of $g^{-1}(D)$ are unchanged and $g$ satisfies
\eqref{tt3}. For the last statement, note that $f\vert_{\partial S} =
f_3\vert_{\partial S}$ and $f_4\vert_{\partial S} = g\vert_{\partial
S}$, and the change from $f_3\vert_{\partial S}$ to
$f_4\vert_{\partial S}$ simply reparametrizes along $\partial S$. 
\end{proof}

The map $g$ given by the previous lemma may be quite badly behaved on
$g^{-1}(D)$, with branch points, folding, twisting, etc. This will mostly not
be a concern for us; what matters most is the boundary maps on
$g^{-1}(D)$. The next lemma achieves tautness of these boundary maps
at the expense of creating branch points over $\Sigma_1$. 

\begin{lemma}\label{tautlemma2}
Suppose $g\co S \to \Sigma$ is an admissible map which satisfies
the conclusions of Lemma~\ref{tautlemma1}. Then there is an
admissible map $g_1\co S_1 \to \Sigma$ with $n(S_1) = n(S)$ and
$\chi(S) \leq \chi(S_1)$ such that 
\begin{enumerate}
\item \label{T1} $g_1 \co g_1^{-1}(\Sigma_2) \to \Sigma_2$ is a branched
immersion with all branch points in the interior of $g_1^{-1}(\Sigma_1)$, 
\item \label{T2} $g_1^{-1}(D)$ is a disjoint union of disks,
\item \label{T3} for each component $E$ of $g_1^{-1}(D)$, the boundary map
$g_1|_{\partial E}$ is a taut turn path,
\item \label{T4} $g_1$ is an immersion on every component of
$g_1^{-1}(\Sigma_3)$ that does not meet $\partial S_1$.
\end{enumerate}
\end{lemma}

\begin{proof}
Let $E$ be a component of $g^{-1}(D)$. There are two types of moves
we will use to achieve tautness in item \eqref{T3}. Suppose first that
$\partial E$ contains two or more instances of the turn arc
$\tau_i$. Delete the interior of $E$, cut the midpoints of the two
instances of $\tau_i$ and glue the first half of one to the second
half of the other and vice versa. This splits $\partial E$ into two
components, which can be filled by two disks mapping to $D$,
increasing $\chi(S)$. After finitely many moves of this type, there
are no repeated instances of turn arcs in components of
$\partial g^{-1}(D)$. The new boundary of $S$ still maps by positive
powers of $\gamma$ with degree $n(S)$. 

Now suppose that the arc $\alpha \in \{\alpha_i\} \cup \{
\overline{\alpha}_i\}$ appears two or more times in $\partial
E$. These two occurrences are also part of $\partial
g^{-1}(\Sigma_1)$; denote them by $\alpha'$ and $\alpha''$. Let $U$ be
an open neighborhood of $\alpha$ in $\Sigma_1$ that is evenly covered
by $g$, and let $U'$, $U''$ be the corresponding neighborhoods of
$\alpha'$ and $\alpha''$ in $g^{-1}(\Sigma_1)$. Let $\beta$
be a small arc in $U$ from the midpoint of $\alpha$ to an interior
point of $U$. Let $\beta'$, $\beta''$ be the lifts of $\beta$ to $U'$
and $U''$ respectively. Now delete the interior of $E$ and cut along
$\beta'$ and $\beta''$. Rejoin the left side of $\beta'$ to the right
side of $\beta''$ and vice versa. This creates an index two branch
point in $g^{-1}(\Sigma_1)$ and splits $\partial E$ into two
components, which can be filled with two disks mapping to $D$. The
Euler characteristic $\chi(S)$ does not change, and neither does
$n(S)$, because this move takes place in the interior of $S$. After
finitely many moves of this second type, we arrive at $g_0 \co S_1
\to\Sigma$ satisfying \eqref{T1}--\eqref{T3}.

Let $E$ be a component of $g_0^{-1}(D)$ which does not meet $\partial
S$. Then the map $g_0\vert_{\partial E}$ is a homeomorphism to
$\partial D'$ for some component $D'$ of $D$ by \eqref{T3}. For each such
$E$, we replace $g_0\vert_E$ by a homeomorphism $E \to D'$ extending
$g_0\vert_{\partial E}$ to get a new map $g_1 \co S_1 \to
\Sigma$. Note that $g_1$ still satisfies \eqref{T1}--\eqref{T3}. By
\eqref{T1} and the above, $g_1$ also satisfies \eqref{T4}. 
\end{proof}

For the last step of the argument, we use finite covers to eliminate
the branch points over $\Sigma_1$. This step depends on $\Sigma_1$
having non-trivial topology. Note that conclusions \eqref{TT1} and
\eqref{TT2} of this lemma are not required for Proposition
\ref{tautprop} but will be useful in a later section. 

\begin{lemma} \label{removebranchpoints}
Suppose $g_1\co S_1 \to \Sigma$ is an admissible map which satisfies
the conclusions of Lemma~\ref{tautlemma2}. Let $r$ be the number of
sheets of the branched cover $g_1 \co g_1^{-1}(\Sigma_1) \to
\Sigma_1$. Then there is a taut admissible map $g_2\co S_2 \to \Sigma$
and a positive integer $N$ such that 
\begin{enumerate}
\item \label{TT1} $n(S_2) = N n(S_1)$,
\item \label{TT2} $\chi(g_2^{-1}(\Sigma_3)) = N \chi(g_1^{-1}(\Sigma_3))$
and $\chi(g_2^{-1}(\Sigma_1)) = N r \chi(\Sigma_1)$,
\item \label{TT3} $\frac{-\chi(S_2)}{2n(S_2)} \leq \frac{-\chi(S_1)}{2n(S_1)}$.
\end{enumerate}
\end{lemma}

\begin{proof}
Let $T_3''$ be the union of all components of $g_1^{-1}(\Sigma_3)$
which meet $\partial S_1$. Let
$$  d = \lcm\{\text{degree of } C' \overset{g_1}{\to} C \mid C', C
\text{ are components of } \partial T_3'', \partial \Sigma_3 \text{
  resp.}\}$$ 
By Lemma \ref{coverlem2}, there is a finite-sheeted cover $T_3' \to
T_3''$ such that every component of $\partial T_3'$ covers a component
of $\partial \Sigma_3$ with degree $\pm d$ under the composition $T_3'
\to T_3'' \overset{g_1}{\to} \Sigma_3$. (The orientation of a covering
surface will be induced by the base unless indicated otherwise.) Let
$\Sigma_{3, 1}, \dots, \Sigma_{3, \ell}$ be the components of
$\Sigma_3$. By Lemma \ref{coverlem}\eqref{cover2}, there are connected
finite-sheeted covers $T_{3, i}' \to \Sigma_{3, i}$ and $T_1' \to
\Sigma_1$ such that every component of $\partial T_{3, i}'$ and
$\partial T_1'$ covers a component of $\partial \Sigma_3$ with degree
$\pm d$. 
	
Let $s_i^+, s_i^-$ be the positive and negative degrees by which
$g_1^{-1}(\Sigma_{3, i}) - T_3''$ covers $\Sigma_{3, i}$. Let $r^+$
and $r^-$ be the positive and negative degrees respectively of the
branched cover $g_1 \co g_1^{-1}(\Sigma_1) \to \Sigma_1$. By taking
multiple copies of $T_{3, i}'$ for all $i$ and multiple copies of
$T_3'$ and $T_1'$, and by orienting the copies of $T_{3, i}'$ and
$T_1'$ appropriately, we can find covers $h_3 \co T_3 \to T_3''$, 
$h_{3, i} \co T_{3, i} \to \Sigma_{3, i},$ and $h_1 \co T_1 \to
\Sigma_1$ such that, for some positive integer $N$,
\begin{itemize}
\item the positive (resp. negative) degree of $h_{3, i}$ is $N s_i^+$
(resp. $N s_i^-$),
\item $h_3$ has degree $N$,
\item $h_1$ has positive degree $N r^+$ and negative degree $N r^-$.
\end{itemize}
	
We now show that we can glue the above pieces together to get a taut
admissible surface $S_2$.  Note that every component of $\partial
\Sigma_{3, i}$ is covered by $\partial T_3''$ with positive and
negative degrees $r^+ - s_i^+$ and $r^- - s_i^-$ respectively.  Thus,
under $g_1 \circ h_3$, every component of $\partial \Sigma_{3, i}$ is
covered by $\frac{N(r^+ - s_i^+)}{d}$ curves by degree $d$ and
$\frac{N(r^- - s_i^-)}{d}$ curves by degree $-d$. Each component of
$\partial \Sigma_{3, i}$, under $h_{3, i}$, is covered by
$\frac{Ns_i^+}{d}$ curves by degree $d$ and $\frac{Ns_i^-}{d}$ curves
by degree $-d$. Each component of $\partial \Sigma_{3, i}$, under
$h_1$, is covered by $\frac{Nr^+}{d}$ curves by degree $-d$ and
$\frac{Nr^-}{d}$ curves by degree $d$. Thus, by gluing $T_1$ to $T_3
\cup \bigcup_i T_{3, i}$ along the preimage of $\partial \Sigma_3$
appropriately respecting orientations, we obtain an oriented compact
surface $S_2$ and a map $g_2 \co S_2 \to \Sigma$, which, by
construction, is a taut admissible map.
		
It is clear that $n(S_2) = N n(S_1)$, so it remains to compute the
Euler characteristic of various surfaces. 
\[\begin{array}{rcl} \chi(g_2^{-1}(\Sigma_3)) & = & \chi(T_3) + \sum_i
\chi(T_{3, i}) \\ 
& = & N \chi(T_3'') + \sum_i N(s_i^+ + s_i^-) \chi(\Sigma_{3, i}) \\
& = & N \chi(T_3'') + \sum_i N \chi(g_1^{-1}(\Sigma_{3, i}) - T_3'')
\\ 
& = & N \chi(g_1^{-1}(\Sigma_3)) \end{array}\]
\[\chi(g_2^{-1}(\Sigma_1)) =  N(r^+ + r^-) \chi(\Sigma_1)  
= N r \chi(\Sigma_1) \]
Finally, note that $\chi(g_1^{-1}(\Sigma_1)) \leq r \chi(\Sigma_1)$
since $g_1 \co g_1^{-1}(\Sigma_1) \to \Sigma_1$ is a branched cover,
and so
\[\chi(S_2) = \chi(g_2^{-1}(\Sigma_3)) + \chi(g_2^{-1}(\Sigma_1)) = N
\chi(g_1^{-1}(\Sigma_3)) + N r \chi(\Sigma_1) \geq N \chi(S_1).\]
\end{proof}

\begin{proof}[Proof of Proposition~\ref{tautprop}] 
Start with an admissible surface $f \co S \to \Sigma$. We may
assume that $f$ is incompressible, since performing a compression
reduces $-\chi(S)$ without changing $n(S)$ or the property of
being admissible. Now apply Lemmas \ref{tautlemma1}, \ref{tautlemma2},
and \ref{removebranchpoints}. 
\end{proof}

\section{Encoding surfaces}
Let $\TP$ denote the set of all taut turn paths. For each taut turn
path $p \in \TP$, we define the sides of $p$. Recall that $p\co S^1 \to
D$ is an immersed path which alternates between turn arcs and
immersed paths in $\partial D$. Since the intersection of $\gamma$
with $\Sigma_2$ lies in $M$, the turn path $p$ alternates between arcs from
$\{\tau_i\} \cup \{\alpha_i\} \cup \{\overline{\alpha}_i\}$ and arcs
in $\partial M - \Sigma_1$. A \textit{side} of $p$ is an ordered pair
$s \in (\{\tau_i\} \cup \{\alpha_i\{\cup\{\overline{\alpha}_i\})^2$
such that $p$ traverses the first arc in the ordered pair, then a part
of $\partial M - \Sigma_1$, and then the second arc in the ordered
pair. For any side $s$, we let $m_s \subseteq \partial M$ denote the
middle arc. See Figure \ref{sides}. Note that since $p$ is taut, it
traverses any side at most once. 

Let $\sides$ denote the set of all ordered pairs of arcs from
$\{\tau_i\} \cup \{\alpha_i\} \cup \{\overline{\alpha}_i\}$ which are
a side for some taut turn path $p$. Note that $\sides$ does not contain all
ordered pairs. E.g., $\sides$ cannot contain any pairs of the form
$(\alpha_i, \overline{\alpha}_j)$ since a taut turn path maps as an
immersion to $\partial D$. E.g., any pair $(\tau_i, \tau_j) \in \sides$
must have one arc oriented into $M$ and one oriented out from $M$ on
the same side of $M$. 

For any side $s \in \sides$, there is a unique side $\hat s \in \sides$ such
that $m_s, m_{\hat s}$ and two unique arcs from $(\gamma \cup \partial
\Sigma_1) \cap M$ bound a subrectangle of $M$. Call $\hat s$ the
\textit{dual side} of $s$, and let $M_s \subseteq M$ be the bounded
subrectangle. Note that, here, we are using the fact that $M$ contains
no self-intersections of $\gamma$. Also, note that $M_s = M_{\hat
s}$. See Figure \ref{sides} again. 

\begin{figure}[!ht]
\includegraphics{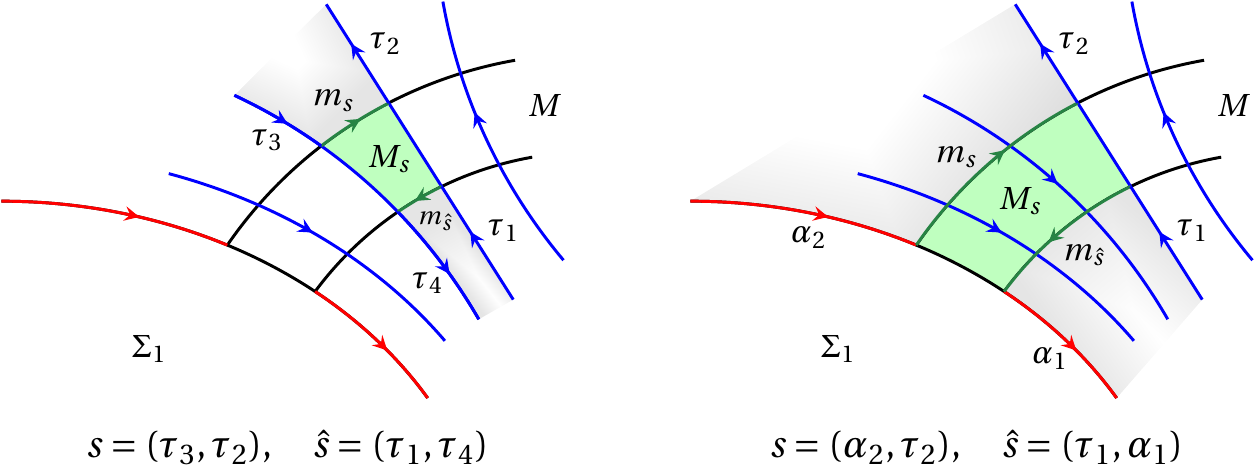}
\caption{Two examples of dual pairs of sides $(s,
\hat{s})$. Geometrically, each side encodes a portion of the
boundary of a possible turn disk (shaded gray in the
figure) mapping to $D$.}\label{sides} 
\end{figure}

\subsection*{The encoding}
Let $V = \R^{\TP} \oplus \R^2$ with coordinates $\{x_p\}, \{r^+,
r^-\}$.  We define a function
\[v\co \Ts(\gamma) \to V\]
as follows. Let $f\co S \to \Sigma$ be a taut admissible map. For each
$p \in \TP$, define $x_p(S)$ to be the number of components $E$ of
$f^{-1}(D)$ such that $f|_{\partial E} = p$. Let $S_1 =
f^{-1}(\Sigma_1)$. As mentioned above, since $f(\partial S) \cap
\Sigma_1 = \emptyset$, the map $f|_{S_1}\co S_1 \to \Sigma_1$ is a
covering map. Define $r^+(S)$ to be the positive degree of
$f\vert_{S_1}$ and $r^-(S)$ the negative degree. Finally let $v(S) =
((x_p(S)), r^+(S), r^-(S))$.

\subsection*{Matching equations}
The vector $v(S)$ will satisfy two kinds of equations. Broadly
speaking, for each turn path $p$ with a side $s$, there must be a
corresponding turn path $p'$ (possibly equal to $p$) with a matching
dual side $\hat s$. Moreover, $r^+$ and $r^-$ should match with the
number of $\alpha_i$'s and $\overline{\alpha}_i$'s appearing in turn
paths of $S$. We define some linear maps. For each side $s \in
\sides$, we define $d_s\co V \to \R$ as
\[d_s((x_p), r^+, r^-)  \ = \sum_{p \text{ has } s} x_p.\] 
For each $\alpha \in \{\alpha_i\} \cup \{\overline{\alpha}_i\}$,
define $d_{\alpha}\co V \to \R$ as
\[d_{\alpha}((x_p), r^+, r^-) \ = \sum_{\substack{\alpha \text{ is a}
    \\  \text{subarc of } p}} x_p.\] 
For each $\tau_i$, define $d_{\tau_i}\co V \to \R$ by 
\[d_{\tau_i}((x_p), r^+, r^-) \ =
\sum_{\substack{\tau_i \text{ is a} \\  \text{subarc of } p}}
x_p.\] 

\begin{lemma} \label{lemma:matchingequations}
Let $(f, S)$ be a taut admissible surface. Then, $((x_p), r^+, r^-) =
v(S)$ satisfies the following matching equations: 
\begin{enumerate}
\item $d_s((x_p), r^+, r^-) = d_{\hat s}((x_p), r^+, r^-)$ for all $s
  \in \sides$. 
\item $d_{\alpha_i}((x_p), r^+, r^-) = r^+$ for all $i$.
\item $d_{\overline{\alpha}_i}((x_p), r^+, r^-) = r^-$ for all $i$.
\item $d_{\tau_i}((x_p), r^+, r^-) = d_{\tau_j}((x_p), r^+, r^-)$ for
  all $i, j$. 
\end{enumerate}
\end{lemma}

\begin{proof}
Let $M'$ be a component of $M$ and let $N'$ be a component of $N =
f^{-1}(M)$ mapping to $M'$. By tautness, the map $f\co N' \to M'$ is
an immersion. Since $M'$ is simply connected, $N'$ is compact, and
$\partial S$ maps exclusively to $\gamma$, the surface $N'$ is a 
topological rectangle with two edges mapping to $\partial D$ and two
edges mapping to $(\gamma \cup \partial \Sigma_1) \cap M$. The map
$f\vert_{N'}$ may be orientation reversing or preserving. 
	
Let $E', E''$ be the two components of $f^{-1}(D)$ bordering
$N'$. Then $N'$ borders $E', E''$ of $E$ in the middle of sides in
$f|_{\partial E'}, f|_{\partial E''}$. Since $f$ is an immersion from
$f^{-1}(\Sigma_2)$ to $\Sigma_2$ and from $\partial S$ to $\Sigma$ (as
a $1$-manifold), the map $f|_{\partial E'}$ (resp. $f|_{\partial
  E''}$) is determined near the intersection of $\partial N'$ and
$\partial E'$ (resp. $\partial N'$ and $\partial E''$). Thus,
$f|_{N'}$ determines precisely one side each of the turn paths
$f|_{\partial E'}, f|_{\partial E''}$ and those sides must be
dual. Conversely, the side of $f|_{\partial E'}$ at $\partial E'
\cap \partial N'$ determines the image of $f(N')$ and the side of
$f|_{\partial E''}$ at $\partial E'' \cap \partial N'$ which must be
dual. Since $f|_{\partial E}$ is taut for any component $E$ of
$f^{-1}(D)$, each arc $\tau_i, \alpha_j, \overline{\alpha}_j$ is
traversed at most once and thus each side appears in $f|_{\partial E}$
at most once. Thus, $d_s(v(S))$ is the total number of times $s$
appears as a side in $f$, and the above argument shows this is equal
to the total number of times $\hat s$ appears as a side in $f$ which
is $d_{\hat s}(v(S))$. 

For the second equality, first note that $r^+(S)$ equals the positive
degree of the map $f \co f^{-1}(\alpha_i) \to \alpha_i$ where we view
$f^{-1}(\alpha_i)$ as a subarc of the boundary of
$f^{-1}(\Sigma_3)$. This is equal to the number of times that the
boundary of $f^{-1}(D)$ traverses $\alpha_i$, and since $f$ is taut,
this is $d_{\alpha_i}(v(S))$. The argument for the third equation is
similar. Again by tautness, $d_{\tau_i}(v(S)) = n(S)$ for all $i$, and
the last equality holds.
\end{proof}

\subsection*{Degree and Euler characteristic}
Recall from the preceding proof that 
$d_{\tau_i}(v(S)) \ = \ n(S)$
for each turn arc $\tau_i$ and all taut admissible surfaces $S$. Thus
we have the factorization 
\[ \Ts(\gamma) \overset{v}{\to} V \overset{d_{\tau_i}}\to \R\]
of the degree function $n(S)$ through $v$ via any of the linear maps
$d_{\tau_i}$. 

The Euler characteristic $\chi(S)$ also factors through $v$ in a
similar fashion. First we need a formula for
$\chi(S)$ based on the combinatorics of turn disks. 

\begin{definition}
For $p \in \TP$ let
\[ \kappa(p) \ = \ 1 - \frac{1}{2}(\# \text{ sides of } p).\]
Let $F_\chi \co \R^{\TP} \oplus \R^2 \to \R$ be the linear map 
\[F_\chi ((x_p), r^+, r^-) \ = \ \Bigl(\sum_{p \in \TP} \kappa(p)
x_p\Bigr) \ + \ \chi(\Sigma_1)(r^+  +  r^-). \]
\end{definition}

\begin{lemma} \label{chiislin}
If $S$ is a taut admissible surface, then  $\chi(S) = F_\chi(v(S))$.
\end{lemma}

\begin{proof}
Given a taut admissible surface $f\co S \to \Sigma$, the subsurface
$S_3 = f^{-1}(\Sigma_2 - \interior(\Sigma_1))$ consists of turn disks 
mapping to $D$ and $1$-handles mapping by immersions to $M$ where each
$1$-handle attaches to the middle arcs of a dual pair of sides. It is
clear then that  
\[ \chi(S) \ = \ \Bigl(\sum_{p \in \TP} \kappa(p) x_p(S)\Bigr).\]
Let $S_1 = f^{-1}(\Sigma_1)$. By definition of taut surface, $S_1$
maps to $\Sigma_1$ by an immersion and since $S_1 \cap \partial S =
\emptyset$, the map is a covering map. Therefore $\chi(S_1) =
\chi(\Sigma_1)(r^+(S)  +  r^-(S))$. Since $S = S_1 \cup S_3$ and $S_1,
S_3$ meet along full boundary components of $S_1$ and $S_3$, we have
$\chi(S) = \chi(S_1) + \chi(S_3)$ and the lemma follows.
\end{proof}
	
\begin{remark}
Having explained our encoding of surfaces, we can relate this encoding
to that used to determine stable commutator length in a free
group. Our $1$--chain $\gamma$ lies in the union of the $D_i$ and
bands $M$. If we collapse each $D_i$ to a point and the bands in $M$
to edges, we obtain a $1$--chain in a graph. However, only the turn
paths which traverse no $\alpha_i$ correspond to turn circuits as
defined in \cite{BCF}, and of course, there is nothing analogous to
the role of $\Sigma_1$ and its matching equations with turn paths. 
\end{remark}

\section{Reassembly}
Given a taut admissible surface, we associate the vector $v(S)$ which
encodes it and satisfies matching equations. Now, we would like to
show that we can construct a taut admissible surface with the right
Euler characteristic and degree ($n(S)$) from a vector satisfying the
matching equations. In general, it's not clear that this is possible
for arbitrary positive integral vectors satisfying the matching
equations. However, it is possible after multiplying the vector by a
sufficiently large integer. In fact, we make the statement for
positive rational vectors. 

\begin{lemma} \label{reconstruction}
Suppose a nonzero, nonnegative vector $((x_p), r^+, r^-) \in \Q^{\TP}
\oplus \Q^2$ satisfies the matching equations (1)-(4) in Lemma
\ref{lemma:matchingequations}. Then, there is a taut admissible
surface $S$ such that $$ \frac{\chi(S)}{n(S)} \ =
\ \frac{F_\chi((x_p), r^+, r^-)}{d_{\tau_1}((x_p), r^+, r^-)}.$$ 
\end{lemma}
\begin{proof}
We can, without loss of generality, assume that the vector is
integral. The basic idea is as follows. We take $x_p$ turn disks for
each $p$ whose boundary will map to the turn path $p$. We then attach
$1$-handles to dual sides. This surface, call it say $S_3$, then has
boundary components some of which map to powers of components of
$\gamma$ and some to components of $\partial \Sigma_3$ with varying
degrees. We then want to glue in a (probably disconnected) surface to
latter boundary components which will map to $\Sigma_1$ by a covering
map. The difficulty is that the part of $\partial S_3$ mapping to
$\partial \Sigma_3 = \partial \Sigma_1$ may not extend to a covering
map of $\Sigma_1$ depending on the topology of $\Sigma_1$ and the
degrees of the maps from components of $\partial S_3$ to components of
$\partial \Sigma_1$. Lemmas \ref{coverlem3} and
\ref{removebranchpoints} will fix this problem at the cost of scaling
$((x_p), r^+, r^-)$ by some positive integer. 
	
For each turn path $p \in \TP$, we let $E_p$ be an oriented disk and
$f_p\co E_p \to D$ a map such that $f_{\partial E_p} = p$ (where
$\partial E_p$ has its boundary orientation). If $p$ does not traverse
any turn arcs, then $p\co S^1 \to \partial D$ is an immersion; in this
case, we insist $f_p$ is a homeomorphism. Note that if $p$ traverses
$\overline{\alpha}_i$'s, then $f_p$ is orientation-reversing. For each
side $s \in \sides$, we let $N_s$ be an oriented disk which is a copy
of $M_s$ and $f_s \co N_s \to M_s$ the canonical homeomorphism. We
orient $N_s$ as follows. If $s$ has a turn $\tau_i$, then we orient
$N_s$ so that the boundary orientation is consistent with the
orientation of $\tau_i$. If $s$ has an $\alpha_i$, then we orient
$N_s$ such that $f_s$ is orientation-preserving and if $s$ has an
$\overline{\alpha}_i$, then we orient $N_s$ such that $f_s$ is
orientation-reversing. 

Now, we construct our initial surface. We take $x_p$ copies of $E_p$
and $d_s((x_p), r^+, r^-)$ copies of $N_t$. (Recall that $M_s =
M_{\hat s}$ and thus $N_s = N_{\hat s}$. By the matching equations, we
are taking a well-defined number of copies of $N_s$.) We attach the
$E_p$ together using the $N_s$ in the obvious way such that $f_p$ and
$f_s$ extend to a map on the oriented surface constructed. Call this
surface $S_3$ and the map $f_3$. Some boundary components of $S_3$ map
to $\gamma$ and others to $\partial \Sigma_3 = \partial \Sigma_1$. By
the matching equations, the positive degree of the map $\partial S_3$
to any component of $\partial \Sigma_3$ is $r^+$ and the negative
degree is $r^-$. By applying Lemma \ref{coverlem3} separately to the
positive and negative degrees, there is a branched cover $S_1 \to
\Sigma_1$ such that the degree partition over any component of
$\partial \Sigma_3$ is precisely the negative of that of $S_3$. By
gluing boundary components of $S_3$ to $S_1$ which map to the same
component of $\partial \Sigma_3$ with matching degrees (i.e. equal but
opposite), we obtain a surface $S'$ and a map $g' \co S' \to \Sigma$
satisfying the hypotheses of Lemma \ref{tautlemma2}.  
	
By Lemma \ref{removebranchpoints}, there is a positive integer $N$ and
a taut admissible surface $g \co S \to \Sigma$ satisfying $n(S) = N
n(S') = N d_{\tau_1}((x_p), r^+, r^-)$ and $\chi(g^{-1}(\Sigma_3)) = N
\chi(S_3)$ and $\chi(g^{-1}(\Sigma_1)) = N(r^+ + r^-)
\chi(\Sigma_1)$. We then have $$\chi(S) =  \chi(g^{-1}(\Sigma_3)) +
\chi(g^{-1}(\Sigma_1)) = N \chi(S_3) + N(r^+ + r^-) \chi(\Sigma_1) = N
F_\chi((x_p), r^+, r^-). $$ The lemma follows.
\end{proof}

\section{Linear programming and scl}
We are now ready to prove that the stable commutator length of the
nonfilling $1$-chain is the solution of a linear program. We define
the region $R \subseteq \R^{\TP} \oplus \R^2$ to be the subset defined
by the following linear equations and inequalities. 
\begin{itemize}
\item the matching equations from Lemma \ref{lemma:matchingequations}
\item $d_{\tau_1}((x_p), r^+, r^-) = 1$
\item $x_p \geq 0$ for all $p \in \TP$ and $r^+ \geq 0$ and $r^- \geq 0$
\end{itemize}

\begin{lemma} \label{sclislinprog}
The stable commutator length of $\gamma$ is equal to 
\begin{equation} \label{linproginf}
\inf\left\{-\frac{1}{2}F_\chi((x_p), r^+, r^-) \mid ((x_p), r^+, r^-)
\in \Q^{\TP} \oplus \Q^2 \cap R \right\}  
\end{equation}
\end{lemma}
\begin{proof}
Recall from Corollary~\ref{tautcor} that 
\[\scl(\gamma) \ = \ \inf_{S \in \Ts(\gamma)}
  \frac{-\chi(S)}{2n(S)}.\] 
Let $S$ be a taut admissible surface. Then, by Lemma
\ref{lemma:matchingequations}, $v(S) \in \Z^{\TP}\oplus \Z^2$
satisfies the matching equations. Moreover, $n(S) = d_{\tau_1}(v(S))$
and, by Lemma \ref{chiislin}, $\chi(S) = F_\chi(v(S))$. By linearity,
$d_{\tau_1}\left(\frac{v(S)}{n(S)}\right) = 1$
and $$-\frac{\chi(S)}{2n(S)} = -\frac{1}{2} \frac{F_\chi(v(S))}{n(S)}
= -\frac{1}{2} F_\chi\left(\frac{v(S)}{n(S)}\right).$$ Thus,
$\scl(\gamma) \geq \eqref{linproginf}$. On the other hand, by Lemma
\ref{reconstruction}, for any vector $((x_p), r^+, r^-)$ satisfying
the matching equations and $d_{\tau_1}((x_p), r^+, r^-) = 1$, there is
a taut admissible surface $S$ such that $$- \frac{\chi(S)}{2n(S)} =
-\frac{1}{2} \frac{F_\chi((x_p), r^+, r^-)}{d_{\tau_1}((x_p), r^+,
  r^-)} = -\frac{1}{2} F_\chi((x_p), r^+, r^-). $$ Thus
$\eqref{linproginf} \geq \scl(\gamma)$. 
\end{proof}

We are now ready to prove the main theorem.

\begin{proof}[Proof of Theorem \ref{mainthm}]
Since $v(S) \in R$ for any taut admissible surface $S$, the region $R$
is non-empty. Since $R$ is defined by equations and inequalities with
rational coefficients, $R \cap \Q^{\TP}\oplus \Q^2$ is dense in
$R$. By Lemma \ref{sclislinprog} and the fact that stable commutator
length is nonnegative, $-\frac{1}{2} F_\chi$ is bounded below on
$R$. Since $R$ is finite-dimensional and defined by finitely many
equations, minimizing the linear function $-\frac{1}{2} F_\chi$ on $R$
is a linear program and therefore has an optimal solution at a vertex
of $R$. Since $R$ is defined by equations with rational coefficients,
all vertices are rational points, and so by Lemma \ref{sclislinprog},
$\scl(\gamma)$ is rational.
\end{proof}

\begin{remark}
The proofs of Theorem~\ref{mainthm} and Lemma \ref{reconstruction}
also establish that $\gamma$ has an extremal surface. 
\end{remark}

\bibliographystyle{amsalpha}
\bibliography{SCL}

\end{document}